\documentclass[12pt]{article}

\usepackage{latexsym}
\usepackage{amsmath}
\usepackage{amssymb}
\usepackage{amsthm}
\usepackage{amscd}
\usepackage[mathscr]{eucal}
\usepackage{fullpage}
\usepackage{relsize}
\usepackage{enumerate}
\usepackage{color}
\newcommand{\R}{\mathbb{R}}

\newcommand{\ed}[1]{#1}

\newtheorem{Theorem}{Theorem}[section]
\newtheorem{Lemma}[Theorem]{Lemma}
\newtheorem{Proposition}[Theorem]{Proposition}
\newtheorem{Corollary}[Theorem]{Corollary}
\newtheorem{Definition}[Theorem]{Definition}

\author{Simone Calogero\\Department of Mathematics\\ Chalmers Institute of Technology\\ University of Gothenburg\\ Gothenburg, Sweden\\ calogero@chalmers.se\\[1cm]
Stephen Pankavich \\ Department of Applied Mathematics and Statistics\\ Colorado School of Mines\\ Golden, CO USA\\ pankavic@mines.edu\thanks{S.P. is supported by the US National Science Foundation under awards DMS-1211667 and DMS-1614586.} }
\title{On the spatially homogeneous and isotropic\\ Einstein-Vlasov-Fokker-Planck system\\ with cosmological scalar field}
\date{}

\begin{document}
\maketitle
\abstract{The Einstein-Vlasov-Fokker-Planck system describes the kinetic diffusion dynamics of self-gravitating particles within the Einstein theory of general relativity. We study the Cauchy problem for spatially homogeneous and isotropic solutions and prove the existence of both global solutions and solutions that blow-up in finite time depending on the size of certain functions of the initial data. We also derive information on the large-time behavior of global solutions and toward the singularity for solutions which blow-up in finite time. Our results entail the existence of a phase of decelerated expansion followed by a phase of accelerated expansion, in accordance with the physical expectations in cosmology.}
\section{Introduction}
The purpose of this paper is to study spatially homogeneous and isotropic solutions of the Einstein-Vlasov-Fokker-Planck system. The model describes the kinetic diffusion dynamics of self-gravitating particles within the Einstein theory of general relativity. It is assumed that diffusion takes place in a cosmological scalar field, which can be identified with the dark energy source responsible for the phase of accelerated expansion of the Universe.  

\ed{When relativistic effects are neglected the motion of self-gravitating kinetic particles undergoing diffusion is described by the frictionless Vlasov-Poisson-Fokker-Planck system in the gravitational case, which is given by}
\begin{subequations}\label{VPFP}
\begin{align}
&\partial_tf+p\cdot\nabla_xf-\nabla U\cdot\nabla_pf=\sigma\Delta_pf\\
&\Delta U=4\pi\rho,\quad\rho(t,x)=\int_{\R^3}f(t,x,p)\,dp.
\end{align}
\end{subequations}
Here $f(t,x,p)$ is the phase-space density of a system of unit mass particles, $U(t,x)$ is the Newtonian gravitational potential generated by the particle system, and $\sigma>0$ is the diffusion constant; the remaining physical constants have been set to one.
The mathematical properties of the system~\eqref{VPFP} have been extensively studied in the literature.  In particular, it is known that the Vlasov-Poisson-Fokker-Planck system admits a global unique solution for given initial data~\cite{Bou,Ono}, whose asymptotic behavior for large times has been studied in~\cite{Car}. The proof of both these results makes use of the explicit form of the fundamental solution to the linear Fokker-Planck equation.
Similar results have also been obtained for the related Vlasov-Maxwell-Fokker-Planck system~\cite{RVMFP, VMFP}.

A relativistic generalization of the  Vlasov-Poisson-Fokker-Planck system \ed{in the gravitational case} has been introduced recently in~\cite{C} and the purpose of this paper is to initiate its mathematical study. When dealing with the relativistic model one is faced with many new difficulties, including the hyperbolic and nonlinear character of the Einstein field equations (compared to the linearity of the Poisson equation in~\eqref{VPFP}), the non-uniform ellipticity of the diffusion operator, the time-dependence of the diffusion matrix, and the absence of an explicit formula for the fundamental solution to the \ed{linear} relativistic Fokker-Planck equation. The further notorious complexity of the Einstein equations suggests that one begins by studying solutions with symmetries. In this paper we consider solutions which are spatially homogeneous and isotropic. Under these symmetry assumptions the Einstein equations reduce to a system of nonlinear ordinary differential equations.

In the absence of diffusion our model reduces to the Einstein-Vlasov system with cosmological constant. The regularity and asymptotic behavior of spatially homogeneous solutions to the latter system have been studied in~\cite{HL}. The generalization of these results to spatially inhomogeneous solutions (with surface symmetry) is given in~\cite{Sop}. When the cosmological constant is set to zero one obtains the Einstein-Vlasov system. \ed{Some} results available for this system are summarized in~\cite{And}. \ed{It has been shown recently in~\cite{FJS,LT} that the Cauchy problem for the Einstein-Vlasov system is globally well-posed for small data with no symmetry restrictions.}

In~\cite{ACP} we have studied a similar problem when the Einstein equation is replaced by a nonlinear wave equation for a scalar field. It turns out that for spatially homogeneous and isotropic solutions, the Fokker-Planck equation considered in~\cite{ACP} is the same as the one derived in the present paper and thus the analysis of the matter equation does \ed{not} pose any new difficulty. Contrastingly, the Einstein equations behave quite differently than the field equation considered in~\cite{ACP}

The remainder of the paper is organized as follows. In the next section we derive the Einstein-Vlasov-Fokker-Planck system for spatially homogeneous and isotropic solutions. In Section~\ref{cauchysec} we study the Cauchy problem and prove that, depending on the size of certain functions of the the initial data, regular solutions either exist globally or blow-up in finite time. In Section~\ref{asysec} we derive information on the asymptotic behavior as $t\to\infty$ of global solutions and toward the singularity for solutions which blow-up in finite time. We also show that for an open set of initial data which describe an initially decelerating expanding Universe, global solutions will eventually give rise to a phase of accelerated expansion, in agreement with the expectations in Cosmology~\cite{W}. 
 
\section{Derivation of the model}\label{RW}
We begin with a short description of the general relativistic kinetic theory of diffusion, see~\cite{C} for more details. In the following,  Greek indices run from 0 to 3, while Latin indices run from 1 to 3. Moreover all indices are raised and lowered with the matrix $g_{\mu\nu}$ and its inverse $g^{\mu\nu}:=(g^{-1})_{\mu\nu}$, and the Einstein summation rule is applied (e.g., $g^{\mu\nu}g_{\nu\alpha}=\delta^\mu_{\ \alpha}$), unless otherwise stated. Let $(M,g)$ be a spacetime.  Let $x^\mu$ denote local coordinates on $M$, with $t=x^0$ being timelike, and let $(x^\mu,p^\nu)$ be the induced canonical local coordinates on the tangent bundle of the spacetime. 
The mass-shell for particles with unit mass is the 7-dimensional submanifold of the tangent bundle obtained by imposing $g_{\mu\nu}p^\mu p^\nu=-1$ and $p^0>0$, so that $(p^0,p^1,p^2,p^3)$ can be interpreted as the components of the four-momentum of a unit mass particle. The mass-shell condition is used to express $p^0$ in terms of the spatial components of the four-momentum, namely
\begin{equation}\label{p0}
p^0=-g_{00}^{-1}\big[g_{0j}p^j+\sqrt{(g_{0j}p^j)^2-g_{00}(1+g_{ij}p^ip^j)}\big].
\end{equation}
It follows that $(x^\mu,p^i)$ define a local system of coordinates on the mass-shell. The particle distribution function $f$ is defined on the mass-shell and is therefore a function of the coordinates $(t,x^i,p^j)$. The Fokker-Planck equation for $f$ is the PDE
\begin{equation}\label{FP}
L f:=p^0\partial_t f+p^i\partial_{x^i}f-\Gamma^i_{\ \mu\nu}p^\mu p^\nu\partial_{p^i}f=\sigma\mathcal{D}_p f,
\end{equation}     
where $L$ is the Liouville (or Vlasov) operator, $\Gamma^{\alpha}_{\ \mu\nu}$ denote the Christoffel symbols of the metric $g$, $\sigma$ is the (positive) diffusion constant, and $\mathcal{D}_p$ is the diffusion operator, which is the Laplace-Beltrami operator associated to the Riemannian metric
\begin{equation}\label{h}
h=h_{ij}dp^i dp^j,\quad h_{ij}=g_{ij}+g_{00}\frac{p_ip_j}{(p_0)^2}-\frac{p_i}{p_0}g_{0j}-\frac{p_j}{p_0}g_{0i}.
\end{equation}
We have
\[
\mathcal{D}_pf=\frac{1}{\sqrt{\det h}}\partial_{p^i}\left(\sqrt{\det h}\,(h^{-1})^{ij}\partial_{p^j}f\right),
\]
where $(h^{-1})^{ij}$ is the inverse matrix of $h_{ij}$ (i.e., $(h^{-1})^{ij}h_{jk}=\delta^i_{\ k}$).
From the particle distribution $f$ we can construct the particle current density and the energy-momentum tensor by 
\begin{equation}\label{T}
T^{\mu\nu}=\sqrt{-\det g}\int f\, \frac{p^\mu p^\nu }{-p_0}\,dp,\quad J^{\mu}=\sqrt{-\det g}\int f\,   \frac{p^\mu}{-p_0}\,dp,
\end{equation}
where $dp=dp^1\wedge dp^2\wedge dp^3$ and the integration is over the fibers of the mass-shell.
It can be shown that the tensors $T^{\mu\nu}$ and $J^\mu$ verify
\begin{equation}\label{equationT}
\nabla_\mu T^{\mu\nu}=3\sigma J^\nu,\quad \nabla_\mu J^\mu=0.
\end{equation}
To close the system we require the metric $g$ to solve the Einstein equations with cosmological scalar field $\phi$, which, in units $8\pi G=c=1$, is given by
\begin{equation}\label{einsteineq}
R_{\mu\nu}-\frac{1}{2}g_{\mu\nu}R+\phi g_{\mu\nu}=T_{\mu\nu}.
\end{equation}
By~\eqref{equationT} and the Bianchi identity $\nabla^\mu (R_{\mu\nu}-\frac{1}{2}g_{\mu\nu}R)=0$, we find that the cosmological scalar field satisfies the equation
\begin{equation}\label{eqphi}
\nabla_\mu\phi=3\sigma J_\mu.
\end{equation}
In particular, $\phi$ satisfies the wave equation $\Box\phi=\nabla^\mu\nabla_\mu\phi=3\sigma\nabla^\mu J_\mu=0$ and
\[
J^\mu\nabla_\mu\phi=3\sigma J_\mu J^\mu<0,
\]
where for the inequality we have used the fact that $J^\mu$ is timelike. Hence, the cosmological scalar field is decreasing along the matter flow, which can be interpreted as energy being transferred from the scalar field to the particles by diffusion.

Next, we specialize to spatially homogeneous and isotropic solutions. Under these symmetry assumptions, the spacetime metric can be written as 
\begin{equation}\label{metricRW}
g_{\mu\nu}dx^\mu dx^\nu=-dt^2+\frac{a(t)^2}{\mathcal{K}(r)^{2}}\delta_{ij}dx^idx^j,
\end{equation}  
where 
\[
\mathcal{K}(r)=1+\frac{k}{4}r^2,\quad r=|x|,\quad k\in\{-1,0,1\}.
\]
Here $x=(x^1,x^2,x^3)$ is a system of spatial isotropic coordinates on the hypersurfaces $t=const.$, with $t$ denoting the proper time along the normal geodesics, and $|\cdot|$ denoting the standard Euclidean norm. The hypersurfaces of constant proper time have zero/positive/negative constant curvature according to the values $k=0,+1,-1$ of the curvature parameter $k$. Equation~\eqref{p0} gives
\begin{equation}\label{p0new}
p^0=\Bigg[1+\frac{a(t)^2}{\mathcal{K}(r)^{2}}|p|^2\Bigg]^{1/2},\quad p=(p^1,p^2,p^3)
\end{equation}
and $p_0=-p^0$. The non-zero Christoffel symbols $\Gamma^i_{\ \mu\nu}$ of the metric~\eqref{metricRW} are given by
\[
\Gamma^{i}_{\ i0}=\frac{\dot{a(t)}}{a(t)},\quad \Gamma^i_{\ ij}=-\frac{k x^j}{2\mathcal{K}(r)},\quad \Gamma^j_{\ ii}=\frac{k x^j}{2\mathcal{K}(r)} \quad  (i\neq j)
\]
and by their symmetric symbols on the lower indexes; the index $i$ is not summed in the previous equations.  It follows that the Liouville operator in the left hand side of~\eqref{FP} is given by
\begin{equation}\label{Lf}
Lf=p^0\partial_t f+p^i\partial_{x^i}f-2\frac{\dot{a(t)}}{a(t)}p^0p\cdot\nabla_p f-\frac{k}{\mathcal{K}(r)}(\tfrac{1}{2}|p|^2x-(x\cdot p)p)\cdot\nabla_p f,
\end{equation}
with $w\cdot z$ denoting the standard Euclidean scalar product of the vectors $w,z\in\R^3$. 
The metric~\eqref{h} now takes the form
\[
h_{ij}=\left(\frac{a(t)^2}{\mathcal{K}(r)^2}\delta_{ij}-\frac{p_ip_j}{(p_0)^2}\right).
\]
It follows that 
\begin{align*}
&\det h=\frac{a(t)^6}{\mathcal{K}(r)^6(p_0)^2}\\
&(h^{-1})^{ij}=g^{ij}+p^ip^j=\frac{\mathcal{K}(r)^2}{a(t)^2}\delta^{ij}+p^ip^j.
\end{align*}
Hence, the diffusion term in the right hand side of~\eqref{FP} ultimately takes the form
\[
\mathcal{D}_pf=p_0\partial_{p^i}\left[\left(\frac{g^{ij}+p^ip^j}{p_0}\right)\partial_{p^j}f\right]=p^0\partial_{p^i}\left[\left(\frac{\frac{\mathcal{K}(r)^2}{a(t)^2}\delta^{ij}+p^ip^j}{p^0}\right)\partial_{p^j}f\right].
\]
\begin{Definition}
A particle distribution $f$ is said to be spatially homogeneous and isotropic if there exists a function $\widetilde{F}:\R\times[0,\infty)\to[0,\infty)$ such that
\[
f(t,x,p)=\widetilde{F}(t,|v|)\biggr \vert_\mathlarger{{v=\frac{a(t)^2}{\mathcal{K}(r)^2}p}}.
\]
\end{Definition}
It can be shown that this definition is equivalent to require that $f$ is invariant by the six-dimensional group of isometries of the metric~\eqref{metricRW}, see~\cite{MM}.
It is now convenient to define the function $F:\R\times\R^3\to [0,\infty)$ by 
\begin{equation}\label{FtildeF}
F(t,v)=\widetilde{F}(t,|v|),
\end{equation}
in terms of which the Fokker-Planck equation for spatially homogeneous and isotropic distribution functions takes the final form
\begin{equation}\label{FPone}
\partial_t F=\sigma a(t)\partial_{v^i}(D^{ij}\partial_{v^j}F),
\end{equation}
where the diffusion matrix $D$ is 
\[
D^{ij}=\frac{a(t)^2\delta^{ij}+v^iv^j}{\sqrt{a(t)^2+|v|^2}}.
\]

In terms of the function $F$, and recalling~\eqref{FtildeF}, the energy-momentum tensor and current density~\eqref{T} read
\begin{align}
&T_{00}(t)=\frac{1}{a(t)^4}\int_{\R^3}F(t,v)\sqrt{a(t)^2+|v|^2}\,dv,\label{T00}\\
&T_{11}(t)=T_{22}(t)=T_{33}(t)=\frac{1}{3a(t)^2\mathcal{K}(r)^2}\int_{\R^3}F(t,v)\frac{|v|^2}{\sqrt{a(t)^2+|v|^2}}\\\label{Tii}
&J^0(t)=\frac{1}{a(t)^3}\int_{\R^3}F(t,v)\,dv,\quad J^i=0,
\end{align}
and $T_{0i}=T_{ij}=0$, for $i\neq j$.  Thus, equation~\eqref{eqphi} for the cosmological scalar field $\phi=\phi(t)$ becomes
 \begin{equation}\label{eqphinew}
 \dot{\phi}=-3\frac{\sigma}{a^3}N,
\end{equation}
where
\begin{equation}\label{numdens}
N=\int_{\R^3}F\,dv
\end{equation}
is the total number of particles, which is conserved along solutions of~\eqref{FPone}.

Finally the non-zero components of the Einstein tensor $G_{\mu\nu}=R_{\mu\nu}-\tfrac{1}{2}g_{\mu\nu}R$ are 
\[
G_{00}=\frac{3(k+\dot{a}^2)}{a^2},\quad G_{11}=G_{22}=G_{33}=-\frac{k+\dot{a}^2+2a\ddot{a}}{\mathcal{K}(r)^2}.
\]
It follows that the Einstein equations~\eqref{einsteineq} are
\begin{align*}
\frac{3(k+\dot{a}^2)}{a^2}-\phi=\frac{1}{a^4}\int_{\R^3}F\,\sqrt{a^2+|v|^2}\,dv,\\
-2\frac{\ddot{a}}{a}-\left(\frac{\dot{a}}{a}\right)^2-\frac{k}{a^2}+\phi=\frac{1}{3a^4}\int_{\R^3}F\,\frac{|v|^2}{\sqrt{a^2+|v|^2}}\,dv.
\end{align*}
By introducing the Hubble function 
\[
H(t)=\frac{\dot{a}(t)}{a(t)},
\] 
as well as the energy density $\rho(t)$ and the pressure $\mathcal{P}(t)$ by
\begin{equation}\label{rhopr}
\rho(t)=\frac{1}{a(t)^4}\int_{\R^3}F\,\sqrt{a(t)^2+|v|^2}\,dv,\quad \mathcal{P}(t)=\frac{1}{3a(t)^4}\int_{\R^3}F\,\frac{|v|^2}{\sqrt{a(t)^2+|v|^2}}\,dv,
\end{equation}
we can rewrite the Einstein equations in the form
\begin{subequations}\label{Einsteinequations}
\begin{align}
&H^2=\frac{1}{3}(\rho+\phi)-\frac{k}{a^2},\\
&\dot{H}=-\frac{\rho+\mathcal{P}}{2}+\frac{k}{a(t)^2}.
\end{align}
\end{subequations}
We also have the following auxiliary equations, obtained by combining \eqref{Einsteinequations}:
\begin{equation}
\label{Hdot}
\dot{H} = -\frac{1}{6} ( \rho + 3\mathcal{P} ) - H^2 + \frac{1}{3}\phi.
\end{equation}
\begin{equation}\label{rhodot}
\dot{\rho}=-3H(\rho+\mathcal{P})-\dot\phi.
\end{equation}
Note also the estimates
\begin{equation}\label{estrhop}
\rho\geq\frac{N}{a^3},\qquad \frac{\rho}{3}-\frac{N}{3a^3}\leq\mathcal{P}\leq\frac{\rho}{3},
\end{equation}
which follow straightforwardly by the definitions~\eqref{numdens} and \eqref{rhopr}.
The deceleration parameter $Q$ is defined as
\begin{equation}\label{q}
Q=-\frac{a\ddot a}{\dot a^2}=-1-\frac{\dot H}{H^2}.
\end{equation}
The solution is said to be expanding with acceleration if $Q<0$ and $H>0$ and expanding with deceleration if $Q>0$ and $H>0$. From~\eqref{Hdot}, this occurs at time $t$ if and only if $H(t)>0$ and 
\begin{equation}\label{Q}
q(t)=6H^2(t)Q(t)=\rho(t)+3\mathcal{P}(t)-2\phi(t)
\end{equation} 
is negative, respectively positive. 
\section{Existence of regular solutions}\label{cauchysec}
In this section we study the existence of regular solutions to the initial value problem for the spatially homogeneous and isotropic Einstein-Vlasov-Fokker-Planck system with cosmological scalar field. The system, derived in the previous section, is given by
\begin{subequations}\label{EVFP}
\begin{align}
&\partial_t F=\sigma a(t)\partial_{v^i}\left(\frac{a(t)^2\delta^{ij}+v^iv^j}{\sqrt{a(t)^2+|v|^2}}\partial_{v^j}F\right),\label{FPeq}\\
&\dot a=Ha\\
&\dot{H}=-\frac{\rho+\mathcal{P}}{2}+\frac{k}{a^2}\label{FRIED}\\
 &\dot{\phi}=-\frac{3\sigma N}{a^3} \label{COSMO},\\ 
&H^2=\frac{1}{3}(\rho+\phi)-\frac{k}{a^2},\label{CONST}
\end{align}
\end{subequations}
where $\rho(t),\mathcal{P}(t)$ are given by~\eqref{rhopr} and $k$ is either $0,+1$, or $-1$. Initial data for the system~\eqref{EVFP} consist of a quadruple $(F_0,a_0,H_0,\phi_0)$, where $F_0:\R^3\to [0,\infty)$ such that $F_0(v)=\widetilde{F}_0(|v|)$, for some $\widetilde{F}_0:\R\to[0,\infty)$, and 
\[
a_0>0, \quad H_0>0,\quad \phi_0>0.
\]
We assume that $F_0$ is not identically zero and belongs to the space $X=\mathcal{L}_1\cap H^1$, where, for $\gamma>0$,
\[
\mathcal{L}_\gamma=\{g:\R^3\to\R:g\in L^1\times L^2,\ \text{and}\ v\to |v|^\gamma g\in L^1\}.
\]
Moreover, the initial data are assumed to satisfy~\eqref{CONST} at time $t=0$, namely
\begin{equation}\label{constin}
H_0^2=\frac{1}{3}(\rho_0+\phi_0)-\frac{k}{a_0^2}.
\end{equation}
Given $T>0$, a quadruple $(F(t,v),a(t),H(t),\phi(t))$ will be referred to as a regular solution of the system~\eqref{EVFP} on the interval $[0,T)$ and with initial data $(F_0,a_0,H_0,\phi_0)$  if 
\[
F\in C^0((0,T);X),\quad 0<a\in C^1((0,T)), \quad H\in C^1((0,T)),\quad \phi\in C^1((0,T)),
\]
\[
F\geq 0 \text{ a.e, }\quad (a_0,H_0,\phi_0)=\lim_{t\to 0^+}(a(t),H(t),\phi(t)),\quad \lim_{t\to 0^+}\|F(t,\cdot)-F_0\|_{\mathcal{L}_1}= 0,
\]
and~\eqref{EVFP} is satisfied in the domain $(t,v)\in (0,T)\times\R^3$. Note that for regular solutions the functions $\rho,\mathcal{P}$ are continuous and hence the Einstein equations are satisfied in the pointwise, classical sense. However the Fokker-Planck equation need only be verified in the weak sense. 

Proving existence and uniqueness of local regular solutions is a simple generalization of the argument presented in~\cite{ACP}. The main tool is the following set of {\it a priori} estimates on the solution of the Fokker-Planck equation. 
\begin{Proposition}\label{estF}
Let $T_\mathrm{max}>0$ be the maximal time of existence of a regular solution. The following estimates on $F$ hold for all $t\in [0,T_\mathrm{max})$:
 
\begin{itemize}
\item[(i)] $N = \| F(t)\|_{L^1}=\|F_0\|_{L^1}$ (conservation of the total number of particles) and
$\| F(t)\|_{L^2}\leq \|F_0\|_{L^2}$ (dissipation estimate).

\item[(ii)] Propagation of moments: if $F_0\in\mathcal{L}_\gamma$, then $F(t,\cdot)\in\mathcal{L}_\gamma$ and 
\[
\int_{\R^3}(a(t)^2+|v|^2)^{\ed{\gamma/2}} F(t,v)\,dv\leq e^{C\Big(\alpha t+\int_0^t (H)_+(s)\,ds\Big)}\int_{\R^3}(a_0^2+|v|^2)^{\ed{\gamma/2}} F_0(v)\,dv,
\] 
where $\alpha=\alpha(\sigma)>0$, with $\alpha(0)=0$, and $C>0$ are constants (depending on $\gamma$) and $(y)_+$ denotes the positive part of $y$.

\item[(iii)] Propagation of derivative moments: if $|v|^{\gamma/2}\nabla_v F_0\in L^2$ then
\[
\int_{\R^3}(a(t)^2+|v|^2)^{\ed{\gamma/2}} |\nabla_vF(t,v)|^2\,dv\leq Ce^{C\Big(\alpha t+\int_0^t (H)_+(s)\,ds\Big)}.
\]
\end{itemize}
\end{Proposition}
\begin{proof}
The proof is very similar to the proof of Prop. 2.2 in~\cite{ACP}, hence we limit ourself to formally derive the estimate in (ii). We compute
\begin{align*}
\frac{d}{dt}\int_{\R^3}(a(t)^2+|v|^2)^{\ed{\gamma/2}} F\,dv&=\ed{\gamma} a(t)\dot{a}(t)\int_{\R^3}(a(t)^2+|v|^2)^{\ed{\gamma/2}-1}F\,dv\\
&\quad+\sigma a(t)\int_{\R^3}(a(t)^2+|v|^2)^{\ed{\gamma/2}} \partial_{v^i}(D^{ij}\partial_{v^j}F)\,dv\\
&=\ed{\gamma} H(t) a(t)^2\int_{\R^3}(a(t)^2+|v|^2)^{\ed{\gamma/2}-1}F\,dv\\
&\quad-\sigma\ed{\gamma} a(t)\int_{\R^3}(a(t)^2+|v|^2)^{\ed{\gamma/2}-1}v_iD^{ij}\partial_{v^j}F\,dv\\
&\leq \ed{\gamma} (H)_+(t)\int_{\R^3}(a(t)^2+|v|^2)^{\ed{\gamma/2}}F\,dv\\
&\quad+\sigma\ed{\gamma} a(t)\int_{\R^3}\partial_{v^j}[(a(t)^2+|v|^2)^{\ed{\gamma/2}-1}v_iD^{ij}]F\,dv.
\end{align*}
In the last integral we use the brief calculation
\[
\partial_{v^j}[\dots]=3(a(t)^2+|v|^2)^{\ed{\gamma/2}-1/2}+\left(\ed{\gamma}-1 \right)(a(t)^2+|v|^2)^{\ed{\gamma/2}-3/2}|v|^2\leq C(a(t)^2+|v|^2)^{\ed{\gamma/2}-1/2}.
\]
Hence, we have
\[
\frac{d}{dt}\int_{\R^3}(a(t)^2+|v|^2)^{\ed{\gamma/2}} F\,dv\leq C(\alpha+(H)_+(t))\int_{\R^3}(a(t)^2+|v|^2)^{\ed{\gamma/2}}F\,dv
\]
and Gr\"onwall's Lemma concludes the proof of (ii).

\end{proof}

The following lemma provides a characterization of $\rho(t)$ that will be useful in subsequent results.
\begin{Lemma}
\label{charrho}
For any $t\in[0,T_\mathrm{max})$, we have
$$\rho(t) = a(t)^{-3} \left ( \rho_0 a_0^3 + 3\sigma N t - 3\int_{0}^t H(s) a(s)^3 \mathcal{P}(s) \ ds \right ).$$
In particular, if $H(t) \geq 0$ for all $t \in [0,T_\mathrm{max})$, then
$$\rho(t) \leq (\rho_0 a_0^3 + 3\sigma N t) a(t)^{-3}$$
for $t \in [0,T_\mathrm{max})$.
\end{Lemma}
\begin{proof}
Using \eqref{rhodot} we find
$$\dot{\rho} + 3H\rho = -3H\mathcal{P} - \dot{\phi},$$
and thus using the integrating factor $\exp (\int_0^t 3H(s) \ ds) = \left ( \frac{a(t)}{a_0} \right )^3$, we find
$$\frac{d}{dt} \left ( \rho(t) \left ( \frac{a(t)}{a_0} \right )^3 \right ) = (-3H(t)\mathcal{P}(t) - \dot{\phi}(t) ) \left ( \frac{a(t)}{a_0} \right )^3.$$
Integrating over $[0, t]$ and using \eqref{COSMO}, this becomes
$$\rho(t) a(t)^3 - \rho_0a_0^3 = 3\int_{0}^t \left ( -H(s)a(s)^3\mathcal{P}(s) + \sigma N \right ) \ ds$$
and the conclusion follows.
\end{proof}

In order to determine whether a regular solution exists globally or blows-up in finite time, we shall often apply the following simple result within subsequent sections.
\begin{Lemma}
\label{Lemmablowup}
\ed{Let $t _1\in \mathbb{R}$ be given and assume there is $t_2 > t_1$ with $I\in C^1(t_1,t_2)$ satisfying $I_1 := I(t_1) < 0$ and
$$\dot{I}(t) \leq -I(t)^2$$
for all $t \in (t_1,t_2)$.
\begin{enumerate}[(a)]
\item If $t_2 \geq t_1 - \frac{1}{I_1}$, then $t_2 = t_1 - \frac{1}{I_1}$ and 
$\lim_{t \to t_2^-} I(t) = -\infty$ with the estimate
$$I(t) \leq - \frac{1}{t_2 - t}$$
for all $t \in [t_1, t_2).$
\item If $t_2 \leq t_1 - \frac{1}{I_1}$ and $\lim_{t \to t_2^-} I(t) = -\infty$, then
$$I(t) \geq -\frac{1}{t_2-t}$$
for $t \in [t_1, t_2)$.
\end{enumerate}}
\end{Lemma}
\begin{proof}
\ed{Since $\dot{I}(t) \leq 0$ on $(t_1,t_2)$ and $I_1 < 0$, it follows that $I(t) < 0$ on the same interval and we divide by $-I(t)^2$ so that the differential inequality becomes
\begin{equation}
\label{DiffIneq}
\frac{d}{dt} \left ( I(t)^{-1}\right) \geq 1,
\end{equation}
for all $t \in (t_1,t_2)$.
Integrating over $[t_1, \tau]$ for $\tau \in (t_1, t_2)$  yields
$$ \frac{1}{I(\tau)} \geq \tau - \left (t_1 - \frac{1}{I_1} \right)$$ 
and for $\tau \in \left [t_1, t_1 -  \frac{1}{I_1}\right )$, this can be inverted to find
$$I(\tau) \leq \frac{1}{\tau - \left (t_1 - \frac{1}{I_1} \right)}.$$
Taking the limit as $\tau \to \left (t_1 -  \frac{1}{I_1} \right )^-$ implies $I(\tau) \to -\infty$ at this time.
Finally, since the differential inequality cannot be satisfied after the blow-up time $t_1 - \frac{1}{I_1}$, it follows that $t_2 = t_1 - \frac{1}{I_1}$. The resulting estimate then follows by making this replacement in the above inequality.}

\ed{In the second case, we first note that the limit condition implies $\lim_{t \to t_2^-} I(t)^{-1} = 0$.
Integrating \eqref{DiffIneq} over $[\tau, t_2)$ for $\tau \in (t_1, t_2)$, we find
$$-\frac{1}{I(\tau)} \geq t_2 - \tau$$
and thus the lower bound
$$I(\tau) \geq -\frac{1}{t_2- \tau}$$
for $\tau \in [t_1, t_2)$.}

\end{proof}


\subsection{Characterization of the maximal time of existence}
Next, we focus on proving a series of criteria for the existence of global regular solutions, or their finite-time blowup.
We begin by showing that as long as $a(t)$ remains bounded away from zero, the solution remains regular. 
 \begin{Lemma}\label{nobigcrunch}
 \begin{equation}\label{condglobal}
 \alpha:=\inf_{t\in[0,T_\mathrm{max})}a(t)>0\Rightarrow T_\mathrm{max}=+\infty.
 \end{equation}
 \end{Lemma}
 \begin{proof}
We limit ourselves to showing that, under the assumption within~\eqref{condglobal}, the solution cannot blow-up in finite time. Assume $T_\mathrm{max}<\infty$. Using $a(t)\geq\alpha>0$ in~\eqref{COSMO} and~\eqref{FRIED} gives $\phi\in W^{1,\infty}([0,T_\mathrm{max}))$ and \ed{$H(t)\leq H_0+Ct$. From the latter inequality and Proposition~\ref{estF} (with $\gamma=1$) we conclude that the integral 
\[
\int_{\R^3}\sqrt{a(t)^2+|v|^2}F(t,v)\,dv
\]
is bounded. In particular $F\in L^\infty([0,T_\mathrm{max}),X)$. Moreover using $a(t)\geq\alpha>0$ in~\eqref{rhopr} we find that
 $\rho,\mathcal{P}\in L^\infty([0,T_\mathrm{max}))$. By~\eqref{CONST} and~\eqref{FRIED} we obtain $H\in W^{1,\infty}([0,T_\mathrm{max}))$. As $H=\dot{a}/a$ we also have $a\in W^{1,\infty}([0,T_\mathrm{max}))$ and the proof is complete.  }
 \end{proof}

\begin{Theorem}\label{condblowup}
Let $T_\mathrm{max}>0$ be the maximal  time of existence of a regular solution. Then, for $k=0,-1$ the following are equivalent:
\begin{itemize}
\item[(a)] $\displaystyle \lim_{t\to T_\mathrm{max}^-}\phi(t)\geq0$
\item[(b)] $H(t)>0$, for all $t\in [0,T_\mathrm{max})$
\item[(c)] $T_\mathrm{max}=+\infty.$
\end{itemize}
For $k=1$, we can only prove the weaker statements
\begin{equation}\label{criteria2}
\left .
\begin{gathered}
T_\mathrm{max}=+\infty\Rightarrow\lim_{t\to T_\mathrm{max}^-}\phi(t)\geq0,\\ \lim_{t\to T_\mathrm{max}^-}\phi(t)\geq \min \left \{\frac{4}{N^2}, \frac{9}{4\rho_0a_0^4}\right \} \Rightarrow T_\mathrm{max}=+\infty \text{ and } H(t)>0\text{ for all $t\geq 0$}.
\end{gathered}
\right \}
\end{equation}
\end{Theorem}
\begin{proof}
We first prove the equivalence of the statements (a), (b), (c) in the cases $k = 0, -1$.

(a)$\Rightarrow$(b): Assuming (a) and using the fact that $\phi(t)$ is strictly decreasing on $ [0,T_\mathrm{max})$,  we have $\phi(t)>0$ for all $t\in [0,T_\mathrm{max})$. Hence~\eqref{CONST} gives $H(t)>0$, for all $t\in [0,T_\mathrm{max})$ when $k=0,-1$.  

(b)$\Rightarrow$(c): As $H(t)$ is strictly positive for $t\in [0,T_\mathrm{max})$, then $a(t)$ is increasing and \ed{so $a(t)\geq a_0$ for all $t\in [0,T_\mathrm{max})$}.  Statement (c) then follows immediately by Lemma~\ref{nobigcrunch}.

Next, we prove (c)$\Rightarrow$(a), as well as the analogous statement for $k=1$.
We do so by using the contrapositive $\sim$(a) $\Rightarrow \sim$(c), i.e.~by showing that negative values of $\phi$ imply finite-time blow-up of solutions for $k = -1, 0, 1$.
Assume $\lim_{t \to T_{\max}^-} \phi(t) < 0$.  Then, since $\dot{\phi}(t) < 0$ there is $T_0 \in (0,T_{\max})$ such that $\phi(t) < 0$ for all $t \in [T_0,T_{\max}).$
With this, \eqref{Hdot} implies 
$$\dot{H}(t) \leq - H(t)^2$$
for all $t \in [T_0,T_{\max})$.
Again using \eqref{Hdot}, we have 
$$\dot{H}(t) \leq \frac{1}{3} \phi(t) \leq \frac{1}{3} \phi(T_0)$$
for all $t \in [T_0, T_{\max})$.
Integrating over $[T_0, t]$, we have
\begin{equation}
\label{Hub}
H(t) \leq H(T_0) + \frac{1}{3}\phi(T_0) (t - T_0)
\end{equation}
for $t \in [T_0, T_{\max})$.
Finally, because $\phi(T_0) < 0$, if $T_{\max} = \infty$ we may take $t$ large enough in \eqref{Hub} so that $H(T_1) < 0$ where $T_1 > \max \left \{ T_0, T_0 - \frac{3H(T_0)}{\phi(T_0)} \right \}$.
Applying Lemma \ref{Lemmablowup} at $t = T_1$, this implies $T_{\max} < \infty$ contradicting the assumption $T_{\max} = \infty$.  Thus, in all cases we find $T_{\max} < \infty$.

%

Finally, we complete the proof by showing that either positive lower bound on $\phi$ in the case $k=1$ implies global existence.
First, assume $$\lim_{t\to T_\mathrm{max}^-}\phi(t)\geq\frac{4}{N^2}.$$
As before, since $\phi$ is decreasing, we have $\phi(t) > \frac{4}{N^2}$ for all $t \in [0,T_{\max})$.
Using the lower bound $\rho(t)\geq N/a(t)^3$ in~\eqref{CONST} we obtain
$$H(t)^2 > \frac{1}{3} \left ( \frac{N}{a(t)^3} + \frac{4}{N^2}\right) - \frac{1}{a(t)^2}.$$
Defining the function
$$g(x) = \frac{1}{3} \left ( \frac{N}{x^3} + \frac{4}{N^2} \right) - \frac{1}{x^2}$$
for all $x > 0$, we find
$$g'(x) = \frac{-N + 2x}{x^4}.$$
Therefore, $g'(x) = 0$ only at $x = \frac{N}{2}$, and $g$ is minimized at this point.
Thus,
$$H(t)^2 > g(a(t)) \geq g \left (\frac{N}{2} \right ) = 0.$$
This implies that $H$ does not change sign and thus $H(t) > 0$ on $[0,T_{\max})$.
Hence, $\dot{a}(t) > 0$, which implies \ed{$a(t)\geq a_0$ for all $t\in [0,T_\mathrm{max})$}, and $T_{\max} = \infty$ follows from Lemma~\ref{nobigcrunch}.

Now, if we instead assume $$\lim_{t\to T_\mathrm{max}^-}\phi(t)\geq\frac{9}{4\rho_0a_0^4},$$
then define
$$T^* = \sup \left \{t \geq 0 : H(s) > 0 \text{ for all } s\in [0,t] \right \},$$
and note that $T^* > 0$ because $H_0 > 0$.
For $t \in [0,T^*)$ we use \eqref{rhodot} to find
$$\dot{\rho} = -3H(\rho + \mathcal{P}) - \dot{\phi} \geq -4H\rho,$$
and thus
$$\frac{d}{dt} \left [a(t)^4 \rho(t) \right ] \geq 0.$$
Hence, a lower bound on $\rho$ follows, namely
$$\rho(t) \geq \rho_0a_0^4 a(t)^{-4}.$$
Thus, on the same time interval, we use \eqref{CONST} to find
$$H(t)^2 \geq \frac{1}{3} \phi(t) + \frac{1}{3} \rho_0 a_0^4 a(t)^{-4} - a(t)^{-2}.$$
Defining the function
$$h(x) = \frac{1}{3} \rho_0 a_0^4 x^{-4} - x^{-2}$$
for $x > 0$, we find its minimum value occurs when $x^2 = \frac{2}{3} \rho_0 a_0^4$, and at this value
$h(x) = -\frac{3}{4\rho_0a_0^4}$.
Thus, for every $t \in [0,T^*)$
$$H(t)^2 \geq \frac{1}{3} \phi(t) + h(a(t)) \geq \frac{1}{3} \phi(t) - \frac{3}{4\rho_0a_0^4}.$$
Using the above assumption on the strictly decreasing function $\phi(t)$, we see that $H(T^*)$ remains positive, and as in the previous argument this implies $T^* = T_{\max} = \infty$. 
\end{proof}

\subsection{Initial data for global existence and blowup}
Our next goal is to prove that for each $k=0,\pm 1$ there exist conditions on the initial data for which the solution is global and conditions which imply finite time blow-up.  There are many ways to express these conditions; one is to write them in the form $\phi_0<\dots\Rightarrow$ blow-up and $\phi_0>\dots\Rightarrow$ global existence, where the right hand side of the inequality depends on $N,H_0,a_0$ and the given constants $k$ and $\sigma$ (but not on $\phi_0)$.  When expressed in this form it is straightforward to show that the conditions on the initial data derived in the present section are compatible with the constraint equation~\eqref{constin}. We begin with the blow-up results.

\begin{Theorem}
Let $k=0,-1$. Then for 
\begin{equation}
\label{blowup1}
\phi_0<\frac{\sigma N}{H_0 a_0^3},
\end{equation}
the regular solution blows-up in finite time, i.e., $T_\mathrm{max}<\infty$. 
If $k=1$, then the solution blows-up in finite time when
\begin{equation}\label{blowup3}
\phi_0<\frac{9\sigma N^2}{4 a_0^3}\sqrt{\frac{\pi}{2}}\mathrm{Erfc}(x)e^{x^2},\quad x=\frac{9H_0N}{4\sqrt{2}},
\end{equation}
where $\mathrm{Erfc}(z)$ denotes the \ed{complementary} error-function.
\end{Theorem}
\begin{proof}
First, suppose that~\eqref{blowup1} holds and assume $T_\mathrm{max}=+\infty$.  
By~\eqref{FRIED} we have $H(t)\leq H_0$, and hence for all $t \in [0,T_{\max})$ it follows that
$a(t)\leq a_0 e^{H_0t}.$
Integrating~\eqref{COSMO} 
we obtain
\[
\phi(t)=\phi_0-3\sigma N\int_0^t\frac{ds}{a(s)^3}\leq\phi_0+\frac{\sigma N}{a_0^3H_0}(e^{-3H_0t}-1)\to\phi_0-\frac{\sigma N}{a_0^3 H_0}.
\] 
Hence under condition~\eqref{blowup1} we have $\lim_{t\to+\infty}\phi(t)<0$, and so, by Theorem~\ref{condblowup}, $T_\mathrm{max} < \infty$.  This contradicts the original assumption and implies $T_{\max} < \infty$ in the cases $k=0,-1$.


In the case $k=1$, the bound on $H$ is much weaker. 
Hence, let us suppose that ~\eqref{blowup3} holds and assume $T_\mathrm{max}=+\infty$.  
We first notice that by~\eqref{COSMO} and the bound $\rho(t) \geq N/a(t)^3$, 
\[
\dot H\leq \frac{1}{a(t)^2}-\frac{N}{2a(t)^3}\leq \max_{x>0} \left (x^{-2}-\frac{N}{2} x^{-3} \right)=\frac{16}{27 N^2}.
\]
Hence, $H\leq H_0+\frac{16}{27 N^2}t$, which gives
\[
a(t)\leq a_0\exp \left (H_0t+\frac{8}{27 N^2}t^2 \right)=a_0e^{-\frac{27}{32}H_0^2N^2}\exp \left (\frac{1}{N}\sqrt{\frac{8}{27}}\,t+\sqrt{\frac{27}{32}}H_0N \right)^2.
\]
Integrating~\eqref{COSMO} we obtain
\begin{align*}
\phi(t)&\leq \phi_0-\frac{3\sigma N}{a_0^3}e^{\frac{81}{32}H_0^2N^2}\int_0^t\exp \left [-3 \left (\frac{1}{N}\sqrt{\frac{8}{27}}\,s+\sqrt{\frac{27}{32}}H_0N \right)^2 \right]ds\\
&\to \phi_0-\frac{9\sigma N^2}{4 a_0^3}\sqrt{\frac{\pi}{2}}\mathrm{Erfc}(x)e^{x^2},\quad x=\frac{9H_0N}{4\sqrt{2}},
\end{align*}
as $t\to\infty$.
Thus, under condition~\eqref{blowup3} we have $\lim_{t\to+\infty}\phi(t)<0$, and by Theorem~\ref{condblowup} we reach a contradiction. 
\end{proof}
%
With this, we turn to initial data which launch a global-in-time solution.
\begin{Theorem}\label{globalk}
\label{Tge}
Let $k=-1,0$.  Then for
\begin{equation}
\label{ge2}
\phi_0\geq 3\frac{\sigma N}{H_0a_0^3},
\end{equation}
the regular solution is global, i.e., $T_\mathrm{max}=\infty$ and 
\[
\lim_{t\to\infty}\phi(t)\geq \phi_0- 3\frac{\sigma N}{H_0a_0^3}.
\]
\end{Theorem}
\begin{proof}
To show global existence under condition~\eqref{ge2}, we define
$$T^* = \sup \left \{ t \geq 0 : \phi(\tau) > 0  \ \mathrm{for \ all} \ \tau \in [0,t] \right \}.$$   
Note that $T_\mathrm{max} \geq T^*$ by Theorem~\ref{condblowup}. Moreover,~\eqref{CONST} implies that $H(t)>0$ for $t\in [0,T^*)$. By~\eqref{Hdot},~\eqref{CONST} and the bound $\mathcal{P}\leq \rho/3$ we have
\[
\dot H\geq \frac{1}{3}\phi-\frac{1}{3}\rho-H^2=\frac{2}{3}\phi-\frac{k}{a^2}-2H^2\geq-2H^2, \quad t\in [0,T^*).
\]  
Integrating and using the fact that $H(t)>0$ on $[0,T^*)$ we obtain
\[
H(t)\geq\frac{H_0}{1+2H_0t},\quad \text{ for $t\in[0,T^*)$}
\]
and so
\[
a(t)\geq a_0(1+2H_0t)^{1/2},\quad\text{for $t\in[0,T^*$}).
\]
Using~\eqref{COSMO} we get
\begin{equation}\label{noname}
\phi(t)\geq\phi_0-\frac{3\sigma N}{a_0^3}\int_0^t\frac{ds}{(1+2H_0s)^{3/2}}>\phi_0-\frac{3\sigma N}{a_0^3 H_0}.
\end{equation}
Hence, under condition~\eqref{ge2}, if $T^*<T_\mathrm{max}$ we would have $\phi(T^*)>0$, which contradicts the maximality of $T^*$. It follows that $T^*=T_\mathrm{max}=+\infty$. The lower bound on the limit of $\phi(t)$ follows by letting $t\to\infty$ in~\eqref{noname}.
\end{proof}

Notice that the last two results suggest a particular quantity involving initial data that can predict the finite or infinite lifespan of the smooth solution in the cases of $k=-1,0$. More specifically,  if $\frac{\phi_0H_0 a_0^3}{\sigma N}$ is large enough then the solution must be global, while if this quantity is made too small, then solutions must blow-up in finite time.
As before, the situation for $k=1$ is more complicated, and the next result establishes conditions on initial data that guarantee global existence in this case.
\begin{Theorem}\label{globalk1}
Let $k=1$. Then, the solution is global under the conditions
\begin{equation}
\label{gek1a}
\frac{3\sigma}{H_0} \leq 1 \qquad \mathrm{and} \qquad \rho_0 a_0^2 < \frac{3}{2}
\end{equation}
or the conditions
\begin{equation}
\label{gek1b}
\frac{3\sigma}{H_0} > 1 \qquad \mathrm{and} \qquad \rho_0 a_0^2 \leq \frac{3}{2} \left (\frac{H_0}{3\sigma} \right ) \exp \left (1 - \frac{H_0}{3\sigma} \right ).
\end{equation}

\end{Theorem}
\begin{proof}
First, note that under either condition~\eqref{gek1a} or~\eqref{gek1b}, we have $\rho_0 a_0^2<3/2$; hence from \eqref{estrhop} and \eqref{FRIED}, it follows that $\dot{H}(0)>0$.
Define
$$T^*= \sup \left \{ t \geq 0 : H(s) \ \ed{\geq} \ \frac{NH_0}{\rho_0 a_0^3} \ \ \mathrm{for \ all} \ s \in [0,t] \right \}$$
and notice that $T^* > 0$ because $\rho_0 \geq \frac{N}{a_0^3}$ and $\dot{H}(0)>0$.
Then, for $t \in [0,T^*)$ \ed{Lemma}~\ref{charrho} yields
$$\rho(t) \leq (\rho_0 a_0^3 + 3\sigma N t) a(t)^{-3}$$
and similarly, the lower bound
$$ a(t) \geq a_0 \exp \left (\frac{NH_0}{\rho_0 a_0^3} t \right )$$
follows from the bound on $H(t)$. Additionally, on the same time interval \eqref{estrhop} and \eqref{FRIED} then imply
\begin{eqnarray*}
\dot{H}(t) & \geq & -\frac{2}{3} \rho(t)  + \frac{1}{a(t)^2}\\
& \geq & a(t)^{-2} \left [ 1- \frac{2}{3} (\rho_0 a_0^3  + 3\sigma Nt ) a(t)^{-1} \right ]\\
& \geq & a(t)^{-2} \left [ 1- \frac{2}{3} (\rho_0 a_0^3  + 3\sigma Nt ) a_0^{-1} \exp \left (-\frac{NH_0}{\rho_0 a_0^3} t  \right) \right ].
\end{eqnarray*}
Now, the function $$g(t) = 1- \frac{2}{3} (\rho_0 a_0^3  + 3\sigma Nt ) a_0^{-1} \exp \left (-\frac{NH_0}{\rho_0 a_0^3} t  \right)$$ is minimized at the point $$t^* = \frac{\rho_0 a_0^3}{3\sigma N} \left ( \frac{3\sigma}{H_0} - 1 \right ).$$
If $\frac{3\sigma}{H_0} \leq 1$, then $g(t)$ is increasing for all $t \geq 0$ and the minimum for nonnegative times occurs at $t = 0$.  Therefore, 
$$\dot{H}(t) \geq a(t)^{-2} g(0) = a(t)^{-2} \left ( 1- \frac{2}{3}\rho_0a_0^2 \right ) > 0$$
assuming that \eqref{gek1a} holds.
Instead, if $\frac{3\sigma}{H_0} > 1$ then $t^* > 0$.  Therefore, 
$$\dot{H}(t) \geq a(t)^{-2} g(t^*) = a(t)^{-2} \left [ 1- \frac{2\sigma \rho_0 a_0^2}{H_0} \exp \left (\frac{H_0}{3\sigma}  - 1  \right) \right ] > 0$$
assuming that \eqref{gek1b} holds.
In either case, $\dot{H}(t) > 0$ for all $t \in (0,T^*)$, which implies that $H(t) > H_0$ on the same interval.  Hence, the above lower bound on $a(t)$ continues on $[0,T^*)$, yielding a regular solution, and because $H(T^*) > H_0$ we find $T^* = \infty$.
\end{proof}

Our final result in this section is meant to unify the treatment for all three cases.
\begin{Theorem}\label{globalgeneral}
Let $k=-1,0,1$ and take $\beta \in (0,1)$. Then, denoting $k_+ = \max\{k,0\}$, the solution is global under the condition
\begin{equation}
\label{ge1}
\phi_0 \geq \frac{3}{\beta^2 a_0^2} \left [ k_+ + 3\left (\frac{\sigma N}{6} \right)^{2/3}\right ]
\end{equation}
and furthermore
\[
\lim_{t\to\infty}\phi(t)\geq \phi_0-\frac{6}{\beta^2a_0^2}\left(\frac{\sigma N}{6}\right)^{2/3}.
\]
\end{Theorem}
\begin{proof}
Define
$$\gamma = \frac{1}{\beta a_0} \left (\frac{\sigma N}{6} \right )^{1/3} > 0$$ and suppose that~\eqref{ge1} holds. 
After some algebra, we see that this inequality implies
$$ \frac{1}{3}\phi_0 - \frac{\sigma N}{3\beta^3 a_0^3 \gamma} - \frac{k_+}{\beta^2 a_0^2} \geq \gamma^2.$$
Then, define
$$T^* = \sup \left \{ t \in [0,T_{\max}) : a(\tau) \geq \beta a_0 e^{\gamma \tau} \ \mathrm{for \ all} \ \tau \in [0,t] \right \}$$   
and note that since $\beta < 1$, we have $T^* > 0$.
Now, integrating~\eqref{COSMO} for $t \in [0,T^*)$, we find
\[
\phi(t)=\phi_0-3\sigma N\int_0^t\frac{ds}{a(s)^3}\geq \phi_0+\frac{\sigma N}{\gamma \beta^3 a_0^3}(e^{-3\gamma t}-1)\geq \phi_0-\frac{\sigma N}{\gamma \beta^3 a_0^3}.
\] 
Hence, \eqref{CONST} implies
$$H(t)^2 \geq \frac{1}{3} \phi(t) - \frac{k_+}{a(t)^2} \geq \frac{1}{3}\phi_0-\frac{\sigma N}{3\beta^3 a_0^3\gamma} - \frac{k_+}{\beta^2 a_0^2} \geq \gamma^2$$
so that $H(t)$ does not change sign on $[0,T^*)$ and $H(t) \geq \gamma$.  It follows that $$a(t) \geq a_0 e^{\gamma t}$$
for all $t \in [0,T^*)$.  Since $\beta < 1$, this then implies $T^* = T_{\max}$.  \ed{Finally, since $\gamma > 0$ we find
$$\inf_{t \in [0,T_{\max})} a(t) = a_0 > 0.$$ It follows by Lemma~\ref{nobigcrunch} that $T_{\max} = \infty$,} and the proof of the theorem is complete.
\end{proof}

\section{Asymptotic behavior of solutions}\label{asysec}
In this section we analyze the asymptotic behavior as $t\to\infty$ of global solutions and as $t\to T_\mathrm{max}^-$ of solutions which blow-up in finite time. We begin with the former solutions.
\begin{Theorem}
\label{asymp1}
Consider a global regular solution, i.e., $T_\mathrm{max}=\infty$, and let $\displaystyle \phi_{\infty} := \lim_{t \to \infty} \phi(t)\geq0 $. Then for $k=0,-1$ the following holds: 
\begin{itemize}
\item[(i)] If $\phi_{\infty}> 0$, then there exist $C_1,C_2> 0$ such that for all $t \geq 0$
\begin{eqnarray*}
\sqrt{\frac{\phi_{\infty} }{3}}  \leq & H(t) & \leq  \sqrt{\frac{\phi_{\infty} }{3}}  + C_2\exp \left({-\sqrt{\frac{\phi_{\infty} }{3}} t} \right) , \\  
\phi_{\infty} + C_1\sigma \exp \left({-3\sqrt{\frac{\phi_{\infty}}{3}} t} \right ) < & \phi(t) & \leq \phi_{\infty} + C_2 \sigma\exp \left({-3\sqrt{\frac{\phi_{\infty}}{3}} t} \right ), \\
a_0 \exp\left ({\sqrt{\frac{\phi_{\infty} }{3}} t} \right) \leq & a(t) & \leq C_2 \exp\left ({\sqrt{\frac{\phi_{\infty} }{3}} t} \right),\\
C_1\sigma \exp \left({-3\sqrt{\frac{\phi_{\infty}}{3}} t} \right ) \leq & \rho(t) & \leq  C_2 (1+\sigma t) \exp \left({-3\sqrt{\frac{\phi_{\infty}}{3}} t} \right ).
\end{eqnarray*}
\item[(ii)] If $\phi_{\infty}= 0$, then there exist $C_1, C_2 > 0$ such that for all $t \geq 0$
\begin{eqnarray*}
C_1 (1+t)^{-1} \leq & H(t) & \leq C_2(1+t)^{-1/2} \left [1 + \sigma (1+t)^{1/4} \right ], \\  
0 < & \phi(t) & \leq C_2(1 + t)^{-1/2}, \\
C_1(1+t)^{1/2} \leq & a(t) & \leq C_2\exp{ (t^{1/2}[1 + \sigma t^{1/4}}]),\\
C_1\exp{ (-3t^{1/2}[1 + \sigma t^{1/4}}]) \leq & \rho(t) & \leq C_2(1 +t)^{-3/2} (1+\sigma t).
\end{eqnarray*}
\end{itemize}
For the case $k=1$, if $\phi_{\infty}> \min \left \{ \frac{4}{N^2}, \frac{9}{4\rho_0 a_0^4} \right \}:=\phi_m$ then the same conclusion as in (i) holds with $\phi_\infty$ replaced by $\phi_*=\phi_\infty-\phi_m$.
\end{Theorem}

\begin{proof}
We first assume $k=0,-1$. If $\phi_{\infty} > 0$, then, by \eqref{CONST}, 
\begin{equation}
\label{Hlower}
H(t)^2 \geq \frac{1}{3} \phi(t) \geq \frac{1}{3} \phi_{\infty}
\end{equation}
for all $t \geq 0$. Thus $H$ remains strictly positive for all time with
$H(t) \geq \sqrt{\frac{\phi_{\infty}}{3}}.$
It follows that for all $t \geq 0$
\begin{equation}
\label{alower}
a(t) \geq a_0 \exp\left ( \sqrt{\frac{\phi_{\infty}}{3}} t \right ).
\end{equation}
Furthermore, due to \eqref{rhopr} and using the positivity of $H$ in Lemma \ref{charrho} implies
$$N a(t)^{-3} \leq \rho(t) \leq (\rho_0 a_0^3 + 3\sigma N t) a(t)^{-3}$$
which yields the claimed behavior of $\rho$ once the bounds on $a(t)$ are established.
To arrive at a bound on $\phi$, we integrate \eqref{COSMO} and use the \ed{lower} bound on $a$ to find
\begin{eqnarray*}
\phi_{\infty} - \phi(t) & = & -3\sigma N \int_t^{\infty} a(s)^{-3} \ ds\\
& \geq & -\frac{3\sigma N}{a_0^3} \int_t^{\infty} e^{-\sqrt{3\phi_{\infty}}s} \ ds\\
& = & -\sqrt{\frac{3}{\phi_{\infty}}} \frac{\sigma N}{a_0^3} \exp \left (- 3 \sqrt{\frac{\phi_{\infty}}{3}} t \right )
\end{eqnarray*}
for all $t \geq 0$, which implies an upper bound on the behavior of $\phi$. \ed{Using this in \eqref{CONST} with the upper bound on $\rho$ and the lower bound on $a$ implies for all $t \geq 0$}
\ed{$$ H(t)^2 \leq  \frac{1}{3}\phi_\infty + C(1+\sigma t)\exp \left (-3\sqrt{\frac{\phi_{\infty}}{3}}t \right) + a_0^{-2} \exp\left ( -2\sqrt{\frac{\phi_{\infty}}{3}} t \right ).$$}
This further implies
$$\ed{H(t)} 
\ed{\leq \sqrt{\frac{\phi_\infty}{3}} + C\exp \left (-\sqrt{\frac{\phi_{\infty}}{3}}t \right),}$$
and, in conjunction with \eqref{Hlower}, this inequality implies the claimed behavior of $H$.
Finally, this upper bound on $H$ implies an analogous upper bound on $a$, so that
$$a(t) = a_0 \exp \left (\int_0^t H(s) \ ds \right ) \leq C \exp \left (\sqrt{\frac{\phi_\infty}{3}}  t\right).$$
Pairing this with \eqref{alower} implies the bounds on $a(t)$, and the upper bound on $a(t)$ further implies the lower bound on $\phi(t)$ using \eqref{COSMO}.


Next, assume $\phi_{\infty} = 0$. As in the previous case, we find by \eqref{CONST}
$$H(t)^2 \geq \frac{1}{3} \phi(t) > 0$$
for all $t \geq 0$, and thus $H$ remains strictly positive for all time.
As in the proof of Theorem \ref{Tge}, we use \eqref{CONST}, \eqref{Hdot}, and the bound $\mathcal{P}\leq \rho/3$ to find
$$\dot H\geq \frac{1}{3}\phi-\frac{1}{3}\rho-H^2=\frac{2}{3}\phi-\frac{k}{a^2}-2H^2\geq-2H^2$$
for $t\geq 0$.
Integrating and using the positivity of $H(t)$ we obtain
\begin{equation}
\label{Hlower2}
H(t)\geq\frac{H_0}{1+2H_0t}
\end{equation}
and the lower bound
\begin{equation}
\label{alower2}
a(t)\geq a_0(1+2H_0t)^{1/2}
\end{equation}
follows for $t \geq 0$.
Furthermore, the positivity of $H$ implies the same upper bound on $\rho$ as in the previous case, and we again find for $t \geq 0$
$$N a(t)^{-3} \leq \rho(t) \leq C(1+\sigma t) a(t)^{-3}.$$
As before, this yields the claimed asymptotic behavior of $\rho$ once the full behavior of $a$ is obtained.
Integrating \eqref{COSMO} and using \eqref{alower2}, we find
$$ \phi(t) = \phi_{\infty} + 3\sigma N \int_t^\infty a(s)^{-3} \ ds \leq \frac{3\sigma N}{H_0a_0^3} (1 + 2H_0t)^{-1/2}.$$
Next, we use \eqref{CONST}, the upper bounds on $\rho$ and $\phi$, and the lower bound on $a$ so that
\begin{eqnarray*}
H(t)^2 & \leq & \frac{\sigma N}{H_0a_0^3} (1 + 2H_0t)^{-1/2} + C(1+\sigma t) a(t)^{-3} + a_0^{-2}(1+2H_0t)^{-1}\\ 
& \leq & C(1+t)^{-1} \left [1 + \sigma (1+t)^{1/2} \right ]
\end{eqnarray*}
and thus
$$H(t) \leq C(1+t)^{-1/2} \left [1 + \sigma (1+t)^{1/4} \right ].$$
With this, an upper bound on $a(t)$ follows, namely
$$ a(t) \leq C\exp{ (t^{1/2}[1 + \sigma t^{1/4}}]).$$

The last claim for  $k=1$ follows by the same argument used to prove (i) \ed{and the lower bounds on $H$ attained for $k=1$ within the proof of Theorem~\ref{condblowup}}.  For example, assuming $\phi_\infty \geq \frac{4}{N^2}$ the argument proceeds by using
\[
H(t)^2\geq\frac13\Big(\phi_\infty-\frac{4}{N^2}\Big)+\frac{1}{3}\Big(\frac{N}{a(t)^3}+\frac{4}{N^2}\Big)-\frac{1}{a(t)^2}\geq \ed{\frac{\phi_*}{3}}
\]
instead of~\eqref{Hlower}.
\ed{Similarly, assuming $\phi_\infty \geq \frac{9}{4\rho_0a_0^4}$ the argument proceeds by using
\[
H(t)^2\geq\frac13\phi(t)-\frac{3}{4\rho_0a_0^4} \geq \frac{\phi_*}{3}
\]
instead of~\eqref{Hlower}.}
\end{proof} 

We remark that sufficient conditions to have $\phi_\infty>0$ for $k=0,-1$ and $\phi_*>0$ for $k=1$ follow by the lower bounds on $\lim_{t\to\infty}\phi(t)$ derived in Theorems~\ref{globalk} and~\ref{globalgeneral}.
Next, we provide some information on the behavior of solutions that blow-up in finite time as they approach the singularity.
 \begin{Theorem}
\label{asymp2}
Let $k=0,-1$ and consider a regular solution that blows up in finite time, i.e., $T_\mathrm{max} <\infty$. Then, we have
\begin{equation}\label{limits}
\lim_{t \to T_{\max}^-} H(t) = \lim_{t \to T_{\max}^-} \phi(t) = - \infty, \quad \lim_{t \to T_{\max}^-} \rho(t) = \infty, \quad \ \text{and} \ \lim_{t \to T_{\max}^-} a(t) = 0.
\end{equation}
Moreover, the singularity is a curvature singularity, and there exist \ed{$C_1, C_2 > 0$ such that for $t$ sufficiently close to $T_{\max}$
\begin{equation}\label{claimedbounds}
\left. \begin{gathered}
-\frac{1}{T_{\max} - t} \leq H(t) \leq -\frac{C_2}{\sqrt{T_{\max} - t}},\\
-\frac{C_1}{(T_{\max} - t)^2} \leq \phi(t) \leq -\frac{C_2}{\sqrt{T_{\max} - t}},\\ 
C_1(T_{\max} - t) \leq a(t) \leq C_2 \sqrt{T_{\max} - t},\\
\frac{C_1}{(T_{\max} - t)^{3/2}} \leq \rho(t) \leq \frac{C_2}{(T_{\max} - t)^2}. 
\end{gathered}
\quad \right\}
\end{equation}
}For the case $k=1$, if $\displaystyle \lim_{t \to T_{\max}^-} \phi(t) < 0$ then the solution blows up in finite time and the same asymptotic behavior follows with identical bounds.
\end{Theorem}

\begin{proof}
\ed{By Theorem~\ref{condblowup} the finite time blow-up of the solution in the cases $k=0, -1$ implies $H(t) \leq 0$ for some $t \in (0, T_{\max})$ and $\phi(t)$ must eventually attain negative values.  Without loss of generality, we can take $T_1 \in (0,T_{\max})$ large enough so that $\phi(t), H(t) < 0$ for $t \in [T_1, T_{\max})$, as $\phi$ is decreasing and $\phi(t) < 0$ implies $\dot{H}(t) < 0$ by \eqref{Hdot}.
With this, we see that $a(t)$ is strictly decreasing on $(T_1, T_{\max})$, and by Lemma~\ref{nobigcrunch} we find both $a(t) > 0$ for all $t \in [0,T_{\max})$ and $$\inf_{t \in [0,T_{\max})} a(t) = 0.$$ Thus, it follows that $\lim_{t \to T_{\max}^-} a(t) = 0.$ The limits of $H$ and $\rho$ in \eqref{limits} follow from this behavior. Indeed, the limiting behavior of $\rho$ follows directly from the lower bound \eqref{estrhop}. Also, because $\phi(t) < 0$ for $t \in [T_1,T_{\max})$, 
we use \eqref{Hdot} to arrive at
\begin{equation}
\label{Hdotupper}
\dot{H}(t) < - H(t)^2
\end{equation}
for $t \in [T_1, T_{\max})$. 
Then, using the definition of $H$, we compute
$$\dot{H}(t) = \frac{d}{dt} \left (\frac{\dot{a}(t)}{a(t)} \right ) = \frac{\ddot{a}(t)}{a(t)} - H(t)^2$$
so that
$$\ddot{a}(t) = \left (\dot{H}(t) + H(t)^2 \right )a(t) < 0$$
for $t \in [T_1,T_{\max})$. Therefore, $ \dot{a}(t) < \dot{a}(T_1) < 0$ on  $[T_1, T_{\max})$, and 
$$H(t) = \frac{\dot{a}(t)}{a(t)} < \frac{\dot{a}(T_1)}{a(t)}$$ 
for $t \in [T_1, T_{\max})$ then implies $H(t) \to -\infty$ as $t \to T_{\max}^-$. 
The limiting behavior of $\phi$ will be obtained from the upper bound established below.
}

\ed{Next, we turn to establishing the precise asymptotic behavior of these functions at the blow-up time. 
Using \eqref{limits} we find $\lim_{t \to T_{\max}^-} H(t)^{-1} = 0$, and because $T_{\max}$ is the blow-up time, it follows that $T_{\max} \leq T_1 - \frac{1}{H(T_1)}$.
Therefore, using \eqref{Hdotupper} and evoking Lemma \ref{Lemmablowup} with $t_1 = T_1$ and $t_2 = T_{\max}$, we find
$$H(t) \geq - \frac{1}{T_{\max} - t}$$
for $t \in [T_1, T_{\max})$.
}

\ed{With this estimate on $H$ we also find for $\tau \in [T_1,T_{\max})$
$$\dot{a}(\tau) \geq -\frac{1}{T_{\max} - \tau} a(\tau),$$
and upon dividing by $a(\tau) > 0$ and integrating over $[T_1,t]$, the lower bound
$$a(t) \geq \frac{a(T_1)}{T_{\max} - T_1} \left (T_{\max} - t \right ) \geq C\left (T_{\max} - t \right )$$
follows for $t \in [T_1, T_{\max})$.}

\ed{The lower bound on $\phi(t)$ is obtained by using this estimate within~\eqref{COSMO} so that
$$\phi(t) \geq \phi(T_1) - C\sigma \int_{T_1}^t \frac{ds}{(T_{\max} - s)^3} = \phi(T_1) + \frac{C \sigma}{(T_{\max} -T_1)^2} - \frac{C\sigma}{(T_{\max} - t)^2}
\geq - \frac{C}{(T_{\max} - t)^2}$$
for $t$ sufficiently close to $T_{\max}$.
Similarly, to derive the upper bound on $\rho$ we first use \eqref{CONST} to find
$$ \rho = 3 \left ( H^2 + \frac{k}{a^2} \right ) - \phi \leq 3H^2 + \frac{3}{a^2} - \phi.$$ 
Then, the previously-obtained lower bound on $H(t)$ for $t \in [T_1, T_{\max})$ implies $\vert H(t) \vert \leq \frac{1}{T_{\max}- t}$, and thus
\begin{equation}
\label{H2}
H(t)^2 \leq \frac{1}{(T_{\max} - t)^2}
\end{equation}
on the same time interval. Using this with the lower bounds on $a$ and $\phi$, we have 
$$\rho(t) \leq \frac{C}{(T_{\max} - t)^2}$$
for $t$ sufficiently close to $T_{\max}$.
}

\ed{Next, we establish bounds on these functions in the opposite directions. Since we know $a \to 0$ as $t \to T_{\max}^-$ and $H(T_1), \phi(T_1) < 0$ with each of these functions decreasing on  $(T_1, T_{\max})$, we can find $T_2 \in [T_1, T_{\max})$ such that $H(t), \phi(t) < 0$ and $a(t) < 1$ for $t \in [T_2, T_{\max})$. Then, using \eqref{rhodot}, \eqref{estrhop}, and \eqref{COSMO} we find
$$\dot{\rho}(t) \leq -4H(t)\rho(t) + 3\sigma N a(t)^{-3} \leq -4H(t) \rho(t) + 3\sigma N a(t)^{-4}$$
for $t \in [T_2, T_{\max})$.
Multiplying by $a(t)^4$, we can rewrite the inequality as
$$\frac{d}{dt} \left ( \rho(t) a(t)^4 \right ) \leq 3\sigma N$$
and integrating yields
$$\rho(t) \leq a(t)^{-4} \left ( \rho(T_2) a(T_2)^4 + 3\sigma N(T_{\max} - T_2) \right ) \leq Ca(t)^{-4}$$
for $t \in [T_2, T_{\max})$.
Using this estimate within \eqref{CONST}, we find
$$H(t)^2 \leq \frac{1}{3} \rho + a(t)^{-2} \leq C \left (a(t)^{-4} + a(t)^{-2} \right ) \leq C a(t)^{-4}$$
for $t \in [T_2, T_{\max})$.
Multiplying by $a(t)^4$, this becomes
$$\vert a(t) \dot{a}(t) \vert^2 \leq C$$
which, because $\dot{a}(t) < 0$, further implies
$$ -\frac{1}{2} \frac{d}{dt} \left ( a(t)^2 \right) \leq C.$$
Integrating over $[\tau, T_{\max})$ and using the limiting behavior of $a$ as $t \to T_{\max}^-$, we find
$$a(\tau) \leq C \sqrt{T_{\max} - \tau}$$
for $\tau \in [T_2, T_{\max})$.}

\ed{With the upper bound on $a$ established, the lower bound on $\rho$ follows directly from \eqref{estrhop} so that
$$ \rho(t) \geq N a(t)^{-3} \geq \frac{C}{(T_{\max} -t )^{3/2}}$$
for $t \in [T_2, T_{\max})$.
Additionally, the upper bound on $\phi$ follows exactly as before. In particular, integrating~\eqref{COSMO} over $[T_2, t]$ we find
$$\phi(t) \leq \phi(T_2) - C \int_{T_2}^t \frac{ds}{(T_{\max} - s)^{3/2}} \leq - \frac{C}{\sqrt{T_{\max} - t}}$$
for $t$ sufficiently close to $T_{\max}$.
Furthermore, this estimate implies the limiting behavior of $\phi$ within \eqref{limits}.
Finally, the upper bound on $H$ is obtained from \eqref{Hdot} so that
$$ \dot{H}(t) \leq - \frac{1}{6}\rho \leq -\frac{C}{(T_{\max} - t)^{3/2}}$$
for $t \in [T_2, T_{\max})$.
Integrating over $[T_2, \tau]$ then yields the result 
$$H(\tau) \leq -\frac{C}{\sqrt{T_{\max} - \tau}}$$
for $\tau$ sufficiently close to $T_{\max}$. 
}

For the case $k = 1$ we merely repeat this argument under the assumption that $\phi(t) < 0$ for some $t \in (0, T_{\max})$ since this gives rise to negative values of $H(t)$ as in the proof of Theorem~\ref{condblowup}, as well as, the upper bound $\dot{H}(t) \leq - H(t)^2$.

Finally the singularity at $t\to T_\mathrm{max}^-$ is a curvature singularity, for the Einstein equation~\eqref{einsteineq} and the proven bounds imply 
\[
R=4\phi-(\rho+3\mathcal{P})\to -\infty.
\]
\end{proof}

With these results at hand we can finally discuss the important question of the existence of a phase of accelerated expansion of the Universe in the future of a phase of decelerated expansion. There is overwhelming experimental evidence that the Universe is currently expanding with acceleration. Nonetheless, the standard cosmological models require the existence in the past of a phase of decelerated expansion, during which the structures visible today (galaxies, clusters, etc.) have formed. The following last result shows that the diffusion model studied in this paper is able to reproduce this physical behavior.
 
Let us define the dimensionless constants:
\[
\Sigma_0=\frac{\sigma}{H_0},\quad \Phi_0=\frac{\phi_0}{\rho_0}.
\]
\begin{Corollary}
Let $k=0$ or $-1$. Then for $3\Sigma_0<\Phi_0<\frac{1}{2}$ the regular solution is global  and there \ed{exist} $0<t_0<t_1$ such that $q(t)>0$ for $t\in[0,t_0)$, and $q(t)<0$ for $t>t_1$. 
\end{Corollary}
\begin{proof}
As $\rho_0\geq N/a_0^3$, then $\Phi_0>3\Sigma_0$ implies $\phi_0>3\frac{\sigma N}{H_0 a_0^3}$, hence the solution is global by Theorem~\ref{globalk}. Moreover
\[
\lim_{t\to\infty}q(t)=\lim_{t\to\infty} \left [ \rho(t)+3\mathcal{P}(t)-2\phi(t) \right ] =-2\phi_\infty<0,
\]
and therefore there exists $t_1>0$ such that $q(t)$ is negative for $t>t_1$. Furthermore, we find
\[
q(0)=\rho_0+3\mathcal{P}_0-2\phi_0>\rho_0-2\phi_0=\rho_0(1-2\Phi_0)>0.
\]
Hence, $\Phi_0<1/2$ implies $q(0)>0$, which further guarantees the existence of $t_0>0$ such that $q(t)>0$ for $t\in [0,t_0)$.

\end{proof}

\end{document}